\documentclass{article}

\usepackage{arxiv}

\usepackage[utf8]{inputenc} 
\usepackage[T1]{fontenc}    
\usepackage{url}            
\usepackage{booktabs}       
\usepackage{amsfonts}       
\usepackage{nicefrac}       
\usepackage{microtype}      
\usepackage{lipsum}
\usepackage{graphicx}
\graphicspath{ {./images/} }

\usepackage{babel}
\usepackage{amsmath}
\usepackage{amsfonts}
\usepackage{amssymb,amsmath}
\usepackage[T1]{fontenc}  
\usepackage{hhline}
\usepackage{fancyhdr}
\usepackage{lmodern}
\usepackage{mathabx}
\usepackage{mathrsfs}
\usepackage{calrsfs}
\usepackage[hidelinks]{hyperref}
\usepackage{tcolorbox}
\usepackage{calc}
\usepackage{dsfont}
\usepackage{color}
\usepackage{amsthm}
\usepackage{mathtools}
\usepackage{enumitem}
\usepackage{xcolor}
\usepackage{geometry}
\usepackage{bm}
\usepackage{microtype}
\usepackage[round]{natbib}

\newcommand{\mb}{\mathbb}

\newtheorem{theorem}{Theorem}[section]

\theoremstyle{definition}

\newtheorem{remark}{Remark}[section]

\title{Orthogonal parametrisations of Extreme-Value distributions}

\author{
Nathan Huet \\
  Ca’ Foscari University of Venice\\
  Venice, Italy \\
  \texttt{nathan.huet@unive.it}
   \And
 Ilaria Prosdocimi \\
  Ca’ Foscari University of Venice\\
  Venice, Italy \\
  \texttt{ilaria.prosdocimi@unive.it}
}

\begin{document}
\maketitle
\begin{abstract}
Extreme value distributions 
are routinely employed to assess risks connected to extreme events in a large number of applications. 
They typically are two- or three- parameter distributions: the inference can be unstable, which is particularly problematic given the fact that often times these distributions are fitted to small samples. 
Furthermore, the distribution's parameters are generally not directly interpretable and not the key aim of the estimation. 
We present several orthogonal reparametrisations of the main extreme-value distributions, key in the modelling of rare events. In particular, we apply the theory developed in \cite{cox1987parameter} to the Generalised Extreme-Value, Generalised Pareto, and Gumbel distributions. 
We illustrate the principal advantage of these reparametrisations in a simulation study.

\end{abstract}




\section{Introduction}
Being able to appropriately model the distribution of rare events is central in many areas, such as in environmental sciences, to better manage their risks and societal impacts. The two key probabilistic models suited for the analysis of extreme events are the Generalised Extreme-Value (GEV) distribution, to model the maxima over disjoint blocks of observations, and the Generalised Pareto (GP) distribution, to model observations above a high threshold \citep{Coles2001}. The estimation of their parameters is however a recurrent challenge due to their high sensibility to outliers, which is linked to the restrictive sample size inherent at the nature of the problem, and because of the strong interactions between the different parameters. 
We propose in this article a reparametrisation of the GEV distribution based on the framework developed in \cite{cox1987parameter}, whose main advantage is the reduction of correlation between parameter estimates. Similar reparametrisations have been successfully applied to the GP distribution \citep{chavez2005generalized} and the Weibull distribution \citep{hartmann2019laplace}. In both cases, the orthogonal parametrisation has demonstrated improved performance, for example, in terms of faster convergence of MCMC algorithms for approximation of posterior distributions in a Bayesian setting \citep[see, respectively,][]{moins2023bayesian,tanskanen2018aspects}. Deriving these orthogonal parametrisations requires knowing the elements of the Fisher information matrix for the distribution: for extreme values distributions these results are scattered and not all fully available in the literature. In this article we collate all these results in a unique place, providing a useful reference for those seeking to perform inference for extremes.

The article is organized as follows. Section~\ref{sec:background} presents background notions of Extreme-Value Theory and the orthogonal parametrization exposed in \citep{cox1987parameter}. Section~\ref{sec:main} proposes the different reparametrizations for the Gumbel and the $2$-parameter GEV distribution. Section~\ref{sec:simu} illustrates the main properties of the reparametrisations through a simulation study.

\section{Preliminaries}\label{sec:background}

\subsection{Extreme Value distributions}

The probability density function of the GEV distribution is given by
\begin{equation}\label{eq:gev_dens}
    f^{GEV}(x;\mu,\sigma,\xi) = \frac{1}{\sigma}\Big(1+\xi\Big(\frac{x-\mu}{\sigma}\Big)\Big)_+^{-(\xi+1)/\xi}\exp\Big\{-\Big(1+\xi\Big(\frac{x-\mu}{\sigma}\Big)^{-1/\xi}\Big)_+\Big\},
\end{equation}
where $(\mu,\sigma,\xi) \in \mb{R}\times ]0,+\infty[\times \mb{R}$ denote the location, scale and shape parameters, respectively, and for $x \in [\mu - \sigma/\xi,+\infty[$ if $\xi >0$ and for $x \in ]-\infty,\mu - \sigma/\xi]$ if $\xi <0$. The limit case when $\xi \rightarrow 0$ corresponds to the Gumbel distribution whose density is given by
\begin{equation*}
    f^{Gumbel}(x;\mu,\sigma) = \frac{1}{\sigma}\exp\Big(-\Big(\frac{x-\mu}{\sigma}\Big)\Big)\exp\Big\{-\exp\Big(-\Big(\frac{x-\mu}{\sigma}\Big)\Big)\Big\},
\end{equation*}
for $x \in \mb{R}$.

\subsection{Orthogonal parametrisation}\label{sec:ortho_param}
In this paper, the terminology \textit{orthogonal parametrisation} refers to the reparametrisation presented in \cite{cox1987parameter} who propose a general framework yielding orthogonality properties in the Fisher information matrix. For more information about this procedure, and its convenient properties, the reader is invited to consult the original paper, and its rich discussion. Here we recall only some basic notions. 

Suppose that initially a log-likelihood $l^*$ is specified in terms of parameters $(\psi,\theta_1,...,\theta_p):=(\psi,\bm{\theta}) \in \mb{R}\times\mb{R}^p$. The first parameter $\psi$ is typically the parameter of interest which should remain unchanged under reparametrisation. The other parameters are sometimes referred to as nuisance parameters, highlighting the fact that the main object of the inference is $\psi$. 
A reparametrisation $(\psi,\bm{\lambda}) \mapsto \tilde{\bm{\theta}}(\psi,\bm{\lambda})$ is defined through the relation $l^*(\psi,\Tilde{\bm{\theta}}(\psi,\bm{\lambda})) = l(\psi,\bm{\lambda})$, where $l$ denotes the log-likelihood expressed in the new parametrisation. This parametrisation is called \textit{orthogonal} if, for all $j\in{1,...,p}$, the Fisher cross-information between $\psi$ and $\lambda_j$ is zero:
\begin{equation*}
\mb{E}\bigg[\frac{\partial^2l(\psi,\bm{\lambda})}{\partial \psi \partial \lambda_j}\bigg] = 0.
\end{equation*}


\noindent If the transformation $\Tilde{\bm{\theta}}$ has a nonzero Jacobian, then, by the chain rule, this implies that, for all $j \in \{1,...,p\}$

\begin{equation}\label{eq:orthogonal_pde}
    \sum_{i=1}^p i^*_{\Tilde{\theta}_i\Tilde{\theta}_j} \frac{\partial\Tilde{\theta}_i(\psi, \bm{\lambda})}{\partial\psi} = - i^*_{\psi\Tilde{\theta}_j},
\end{equation}
 and where the $i_{\cdot\cdot}^*$'s denote the elements of the Fisher information matrix in the original settings.

\begin{remark}
    The notion of orthogonality considered here, in the sense of \citet{cox1987parameter}, is only local. That is, after the reparametrisation, the parameter $\psi$ becomes orthogonal to the new parameters $\lambda_1,\ldots,\lambda_p$, however there is no guarantee that the parameters $\lambda_1,\ldots,\lambda_p$ are mutually orthogonal.
\end{remark}

\section{Orthogonal reparametrisations for Extreme-Value distributions}\label{sec:main}

We present our main contributions in this section. Our objective is to propose convenient orthogonal reparametrisations for the GEV distribution. Since the derivation of a general orthogonal reparametrisation for the full three-parameter distribution appears to be intractable (see Remark~\ref{rem:3-GEV}), we focus instead on two special two-parameter cases.
In Section~\ref{sec:gumbel}, we derive reparametrisations for the Gumbel distribution, which is the limit distribution of the GEV distribution when $\xi \rightarrow 0$.
In Section~\ref{sec:2-GEV}, we propose to study the special case when $\xi \ne 0$ and the lower or upper finite bound of the distribution is known. This assumption can be justified by physical considerations, as discussed, in, e.g., \cite{zhang2023upper}, in the context of temperature extremes.
Knowledge of this bound reduces the model to a two-parameter GEV distribution \eqref{eq:gev_dens}, for which an explicit orthogonal reparametrisation can be obtained. The orthogonality of the parametrisations provided in this section is illustrated in Section~\ref{sec:simu}. All proofs of the results presented in this section are deferred to Appendix~\ref{sec:appendix_proofs}.


\subsection{The Gumbel distribution}\label{sec:gumbel}
Following the approach presented in Section~\ref{sec:ortho_param}, to obtain a representation with orthogonal parameters for the Gumbel distribution where the scale parameter $\sigma$ is the parameter of interest and the location parameter $\mu$ is transformed, we seek a reparametrisation $(\nu,\sigma) \mapsto (\Tilde{\mu}(\nu,\sigma),\sigma)$ such that
\begin{equation*}
    \Tilde{f}(x;\nu,\sigma) = f(x;\Tilde{\mu}(\nu,\sigma),\sigma),
\end{equation*}
and such that the Fisher information between $\nu$ and $\sigma$ according to $\Tilde{f}$ is equal to zero. 
Hence, the transformation has to satisfy
\begin{equation*}
    i_{\Tilde{\mu} \Tilde{\mu}}\frac{\partial \Tilde{\mu}(\nu,\sigma)}{\partial \sigma} = - i_{\Tilde{\mu}\sigma}.
\end{equation*}
The resulting reparametrisation, as well as the one obtained by taking $\mu$ as the parameter of interest and transforming $\sigma$, are presented in the following theorem.
\begin{theorem}[Orthogonal reparametrisation of the Gumbel distribution]\label{thm:Gumbel}
    An orthogonal reparametrisation of the Gumbel distribution is given as follows
\begin{enumerate}
    \item with transformed location parameter $\mu$,  $(\nu,\sigma) \mapsto (\Tilde{\mu}(\nu,\sigma),\sigma)$, with
    \begin{equation*}
        \Tilde{\mu}(\nu,\sigma) = (1- \gamma)\sigma + C_1(\nu),
    \end{equation*}
    where $\gamma$ denotes the Euler-Mascheroni constant and where $C_1(\nu)$ is independent of $\sigma$.
    \item with transformed scale parameter $\sigma$, $(\mu,\rho) \mapsto (\mu,\Tilde{\sigma}(\mu,\rho)$, with
    \begin{equation*}
        \Tilde{\sigma}(\mu,\rho) = \frac{1- \gamma}{\pi^2/6+\gamma^2 -2\gamma+1}\mu + C_2(\rho),
    \end{equation*}
    where $C_2(\rho)$ is independent of $\mu$.
\end{enumerate}
\end{theorem}


\subsection{The two-parameter GEV distribution}\label{sec:2-GEV}

To reduce the GEV distribution to a two-parameter form, we propose fixing $B$, the lower bound, in case $\xi >0$, or the upper bound, in case $\xi <0$. The case $\xi = 0$ is treated in the previous section. The following derivations are independent of the choice of the value of $B$. For simplicity, we set $B=0$, which fixes the location parameter $\mu = \sigma/\xi$. We denote this distribution by $GEV_2(\sigma,\xi)$. Its density $f^{GEV_{2}}$ is given by
\begin{equation*}
    f^{GEV_{2}}(x;\sigma,\xi) = \frac{1}{\sigma}\Big(\frac{\xi x}{\sigma}\Big)^{-1-1/\xi} \exp\Big\{-\Big(\frac{\xi x}{\sigma}\Big)^{-1/\xi}\Big\},
\end{equation*}
for $x> 0$ if $\xi >0$, and for $x<0$ if $\xi <0$.

To compute the desired parametrisation, the Fisher information of the $GEV_2$ are computed first and presented in \ref{sec:GEV2_fisher}.

No explicit reparametrisation can be derived for the $GEV_2$ distribution when the scale parameter $\sigma$ is kept as the parameter of interest, due to the complex form of $i_{\xi\xi}$. The reparametrisation with the shape parameter $\xi$ unchanged and the scale parameter $\sigma$ transformed is presented in the following theorem.

\begin{theorem}[Orthogonal reparametrisation of the two-parameter GEV distribution]\label{thm:2-GEV} An orthogonal reparametrisation of the $GEV_2$ distribution, with transformed scale parameter $\sigma$, is given by
\begin{equation*}
    \Tilde{\sigma}(\rho,\xi) = C(\rho)\xi\exp\big((1-\gamma)\xi\big),
\end{equation*}
    where $C(\rho)$ is independent of $\xi$.

\end{theorem}


\begin{remark}[The GP distribution]\label{rem:GPD} 
In \cite{chavez2005generalized}, an orthogonal  parametrisation  based on the framework of \cite{cox1987parameter} is employed for the $GP(\sigma,\xi)$ distribution: this is  $(\nu,\xi)\mapsto(\Tilde{\sigma}(\nu,\xi),\xi)$ with $\Tilde{\sigma}(\nu,\xi) = \nu/(\xi+1)$. 
Another orthogonal parametrisation can be obtained with transformed shape parameter $\xi$, given by  $(\sigma,\zeta)\mapsto(\sigma,\Tilde{\xi}(\sigma,\zeta))$, with $\Tilde\xi(\sigma,\zeta)=C_4(\zeta)-\log(\sigma)/2$, where $C_4(\zeta)$ is independent of $\sigma$.

It is interesting to note that, in the three-parameter $GP(\mu,\sigma,\xi)$ case, where $\mu$ is the lower bound of the distribution (or the threshold in the peaks over the threshold approach), the reparametrisation where the shape parameter is the parameter of interest $(\mu,\nu,\xi)\mapsto (\mu,\nu/(\xi+1),\xi)$ is orthogonal in the sense of \cite{cox1987parameter}; that is, the Fisher information between $\xi$ and $\nu$ and between $\xi$ and $\mu$ is zero.

The quantities involved in Fisher information of the GP distribution are recalled in Appendix~\ref{sec:fisher_gpd} and the proofs of the assertions of this remarks are deferred to Appendix~\ref{sec:appendix_proofs}.

\end{remark}

\begin{remark}[The three-parameter GEV distribution case]\label{rem:3-GEV} 

Our initial objective was to derive an orthogonal reparametrisation for the classical three-parameter GEV distribution. However, when moving from a two-parameter to a three-parameter model, the interactions among parameters generally become substantially more complex. This is notably the case for the GEV distribution, whose Fisher information matrix has a particularly intricate structure (see Appendix~\ref{sec:GEV_Fisher}). Moreover, in the presence of three parameters, achieving orthogonality requires transforming two parameters simultaneously, which leads to a system of coupled partial differential equations. For instance, in the GEV case, a reparametrisation that leaves the shape parameter unchanged, $(\nu,\rho,\xi) \mapsto (\Tilde{\mu}(\nu,\rho,\xi),\Tilde{\sigma}(\nu,\rho,\xi),\xi)$, needs to fulfil the equations
\begin{equation*}
    i_{\Tilde{\mu} \Tilde{\mu}} \frac{\partial \Tilde{\mu}(\nu,\rho,\xi)}{\partial \xi} + i_{\Tilde{\mu} \Tilde{\sigma}}\frac{\partial \Tilde{\sigma}(\nu,\rho,\xi)}{\partial \xi} = -i_{\xi \Tilde{\mu}} , \quad
 i_{\Tilde{\mu} \Tilde{\sigma}} \frac{\partial \Tilde{\mu}(\nu,\rho,\xi)}{\partial \xi} + i_{\Tilde{\sigma}\Tilde{\sigma}}\frac{\partial \Tilde{\sigma}(\nu,\rho,\xi)}{\partial \xi} = -i_{\xi \Tilde{\sigma}},
\end{equation*}
which can be rewritten as
\begin{equation*}
    (i_{\Tilde{\mu} \Tilde{\sigma}}^2 - i_{\Tilde{\sigma}\Tilde{\sigma}}i_{\Tilde{\mu} \Tilde{\mu}})\frac{\partial \Tilde{\sigma}(\nu,\rho,\xi)}{\partial \xi} = i_{\xi \Tilde{\sigma}}i_{\Tilde{\mu} \Tilde{\mu}} - i_{\xi \Tilde{\mu}}i_{\Tilde{\mu} \Tilde{\sigma}}, \quad
    (i_{\Tilde{\mu} \Tilde{\sigma}}^2 - i_{\Tilde{\sigma}\Tilde{\sigma}}i_{\Tilde{\mu} \Tilde{\mu}})\frac{\partial \Tilde{\mu}(\nu,\rho,\xi)}{\partial \xi} =  i_{\xi \Tilde{\mu}}i_{\Tilde{\sigma}\Tilde{\sigma}} - i_{\xi \Tilde{\sigma}}i_{\Tilde{\mu} \Tilde{\sigma}}.
\end{equation*}
In view of the expressions of the Fisher information for the GEV distribution given in Appendix~~\ref{sec:GEV_Fisher}, solving these equations appears very challenging, if not impossible, without any approximation.

As already stipulated decades ago by \cite{huzurbazar1956sufficient}, the problem of finding a orthogonal or partially orthogonal reparametrisation  is \textit{in general impossible, except in special cases} (such as the GP distribution). There are well-known connections between the GP and the GEV distributions and their parameters \citep[see][]{Coles2001}. Nevertheless, we find that for only one of them it is possible to derive an orthogonal parametrisation. While some details of the two distributions might be related, performing inference with finite-samples poses different challenges depending on whether the analysis is carried out under the peaks-over-threshold or the block maxima framework.

\end{remark}

\section{Simulation study}\label{sec:simu}
In this section, we illustrate through a simulation study the decorrelation power of the orthogonal parametrisation. 
We simulate $d = 1000$ independent replications of samples of size $n = 100$ for each distribution discussed in this paper, with true parameter values specified using the classical parametrisation. Maximum likelihood estimation is then performed using both parametrisations. The three simulation setups are as follow

\begin{itemize}
    \item in the $GEV_2$ case, the true parameter values are taken to be $(\sigma,\xi)=(1,0.2)$; the orthogonal parametrisation is $(\Tilde{\sigma}(\rho,\xi),\xi)$, using $C(\rho)=\rho$;
    \item in the GP case, the true parameter values are taken to be $(\sigma,\xi)=(1,0.2)$; the orthogonal parametrisations are $(\Tilde{\sigma}(\nu,\xi),\xi)$ using $C_3(\nu)=\nu$, and $(\sigma,\Tilde{\xi}(\sigma,\zeta))$ using $C_4(\zeta)=\zeta$;
    \item in the Gumbel case, the true parameter values are taken to be $(\mu,\sigma)=(1,1)$ the orthogonal parametrisations are $(\Tilde{\mu}(\nu,\sigma),\sigma)$ using $C_1(\nu)=\nu$, and $(\mu,\Tilde{\sigma}(\mu,\rho))$ using $C_2(\rho)=\rho$.
\end{itemize}
For each replication, we compute the cross-correlation between the estimated parameters and display their empirical distribution across the $d$ replications using violin plots (Figure~\ref{fig:cross-correlation}). In addition, we compute the correlation between the vectors of parameter estimates obtained over all replications. These results are reported at the top of Figure~\ref{fig:cross-correlation}.

Although these properties are theoretically valid in the asymptotic regime, Figure~\ref{fig:cross-correlation} shows that the dependence between parameter estimates decreases drastically when one of the proposed reparametrisations is used with a moderate sample size. It further indicates that the cross-correlations of the reparametrised estimates are generally concentrated around zero. The only notable exception occurs in the GP case when using the parametrisation $(\sigma,\Tilde{\xi}(\sigma,\zeta))$. This observation further corroborates the choice of \cite{chavez2005generalized} of using the reparametrisation for the  scale rather than the shape parameter.

\begin{figure}[ht!]
    \caption{Violin plots of the cross-correlation between parameter estimates for the two-parameter GEV distribution (left), the two-parameter GP distribution (middle), and the Gumbel distribution (right), under the different parametrisations considered. The cross-correlations are computed over $d=1000$ independent replications of samples of size $n=100$. 
    \label{fig:cross-correlation}}
\centering
  \includegraphics[width=.3\textwidth]{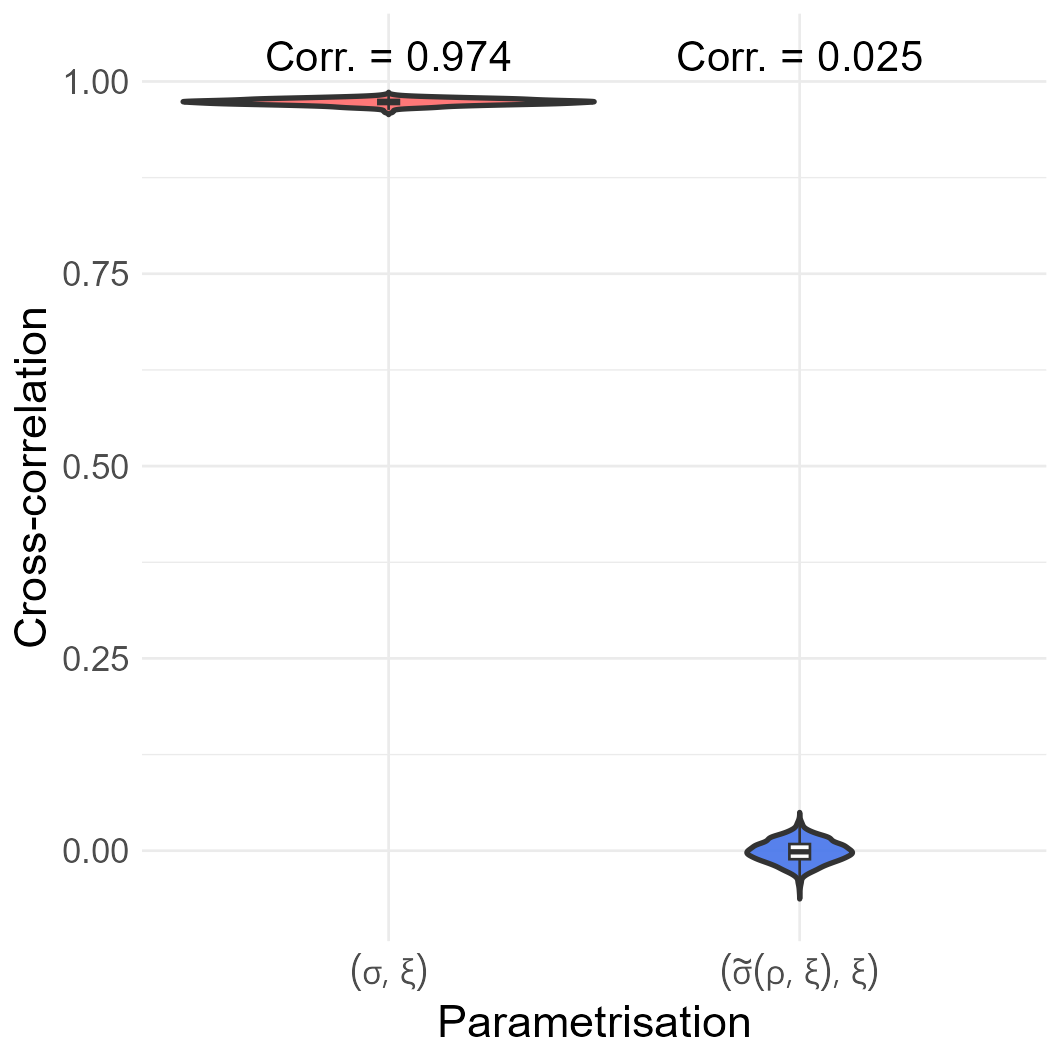}
  \hspace{0.4cm}
  \includegraphics[width=.3\textwidth]{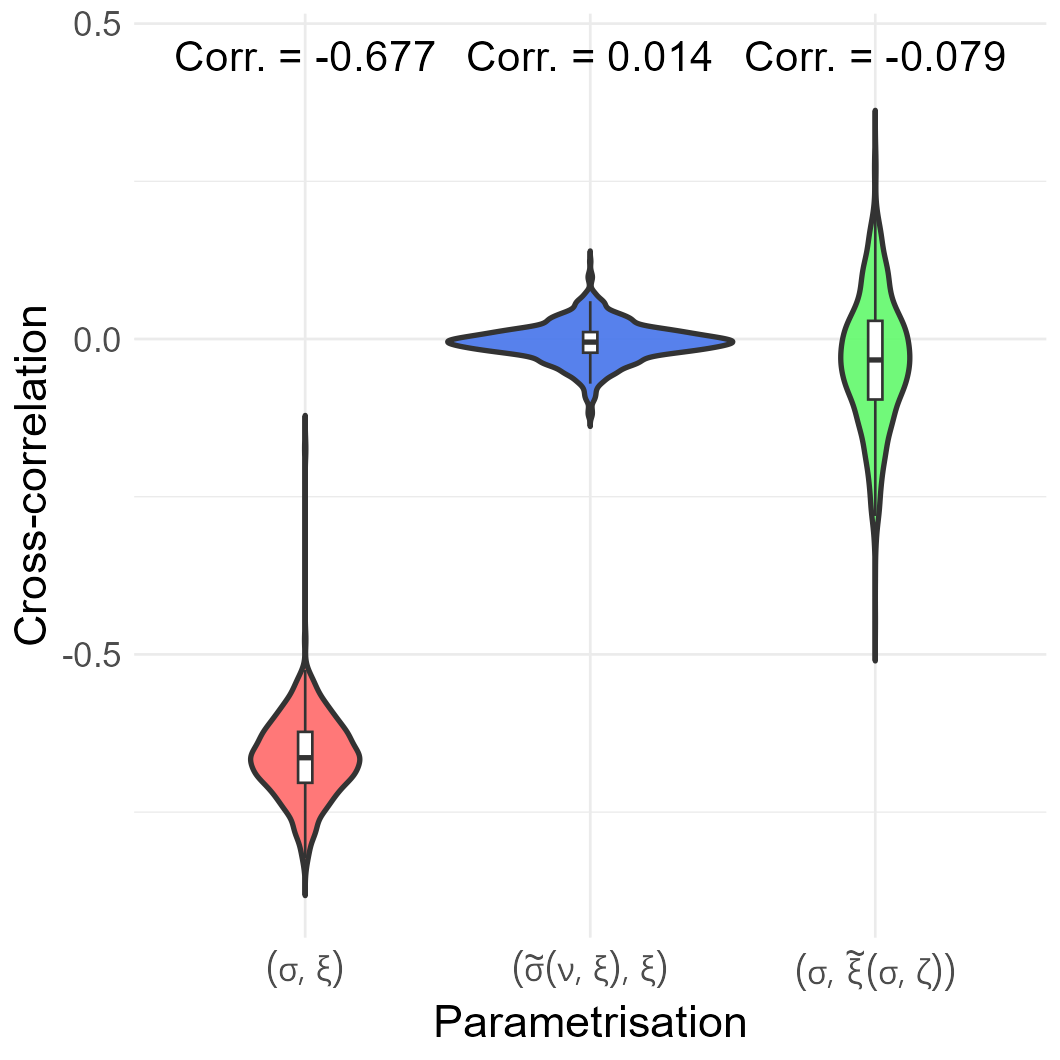}
  \hspace{0.4cm}
  \includegraphics[width=.3\textwidth]{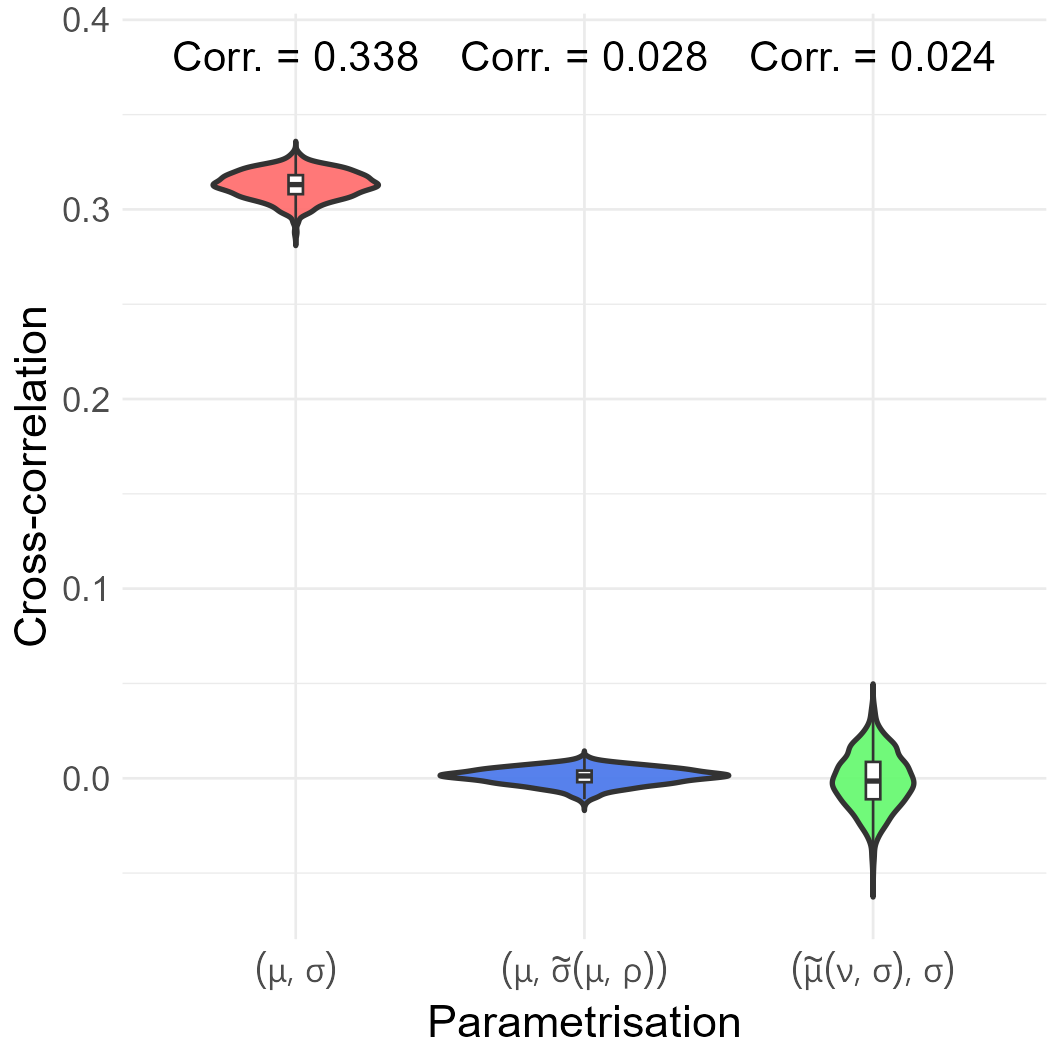}
\end{figure}




\section*{Funding}

This work was supported by DoE 2023-2027 (MUR, AIS.DIP.ECCELLENZA2023\_27.FF project). 


\section*{Data availability}
No data was used in the article. The code for the figures are available at \url{https://github.com/HuetNathan/reparametrisation_GEV}.








\bibliographystyle{abbrvnat}
\bibliography{biblio}
\clearpage
\appendix
\section{Fisher information matrices}

\subsection{Three-parameter GEV distribution}\label{sec:GEV_Fisher}

The Fisher information of the GEV distribution are derived in \cite{prescott1980maximum}. We recall them here for completeness
\begin{equation*}
i_{\mu \mu}=\frac{p(\xi)}{\sigma^2}, \quad i_{\sigma \sigma}=\frac{1}{\sigma^2 \xi^2}(1-2 \Gamma(2+\xi)+p(\xi))
\end{equation*}
\begin{equation*}
i_{\xi \xi}=\frac{1}{\xi^2}\left(\frac{\pi^2}{6}+\left(1-\gamma+\frac{1}{\xi}\right)^2-\frac{2 q(\xi)}{\xi}+\frac{p(\xi)}{\xi^2}\right), \quad i_{\mu \sigma}=-\frac{1}{\sigma^2 \xi}(p(\xi)-\Gamma(2+\xi)),
\end{equation*}
\begin{equation*}
i_{\mu \xi}=-\frac{1}{\sigma \xi}\left(q(\xi)-\frac{p(\xi)}{\xi}\right), \quad i_{\sigma \xi}=-\frac{1}{\sigma \xi^2}\left(1-\gamma+\frac{1-\Gamma(2+\xi)}{\xi}-q(\xi)+\frac{p(\xi)}{\xi}\right),
\end{equation*}
\begin{equation*}
p(\xi)=(1+\xi)^2 \Gamma(1+2 \xi), \quad q(\xi)=\Gamma(2+\xi)\left(\psi(1+\xi)+\xi^{-1}+1\right),
\end{equation*}
with $\Gamma$ the Gamma function, $\psi(z) = \Gamma'(z)/\Gamma(z)$ the digamma function and $\gamma$ the Euler-Mascheroni constant.

\subsection{Gumbel distribution}\label{sec:gumbel_fisher}

    The Fisher information of the $Gumbel(\mu,\sigma)$ distribution are
\begin{equation*}
    i_{\mu\mu} = 1/\sigma^2; \quad i_{\sigma \sigma} = (\pi^2/6+\gamma^2 - 2 \gamma +1)/\sigma^2; \quad i_{\mu \sigma} = (\gamma - 1)/\sigma^2,
\end{equation*}
for $\mu \in \mb{R}$, and $\sigma>0$.

\begin{proof}
    To obtain the Fisher information, we need to compute
    \begin{equation*}
        i_{\mu\mu} = \mb{E} \Big[\Big(\frac{\partial \log(f^{Gumbel}(X;\mu,\sigma))}{\partial \mu}\Big)^2\Big]; \quad i_{\sigma\sigma} = -\mb{E} \Big[\Big(\frac{\partial \log(f^{Gumbel}(X;\mu,\sigma))}{\partial \sigma}\Big)^2\Big];
        \end{equation*}
        \begin{equation*}\label{eq:info_gev}
        i_{\mu\sigma} = -\mb{E} \Big[\frac{\partial \log(f^{Gumbel}(X;\mu,\sigma))}{\partial \mu}\frac{\partial \log(f^{Gumbel}(X;\mu,\sigma))}{\partial \sigma}\Big],
    \end{equation*}
where $X \sim Gumbel(\mu,\sigma)$.

\noindent The scores of the Gumbel distribution are
    \begin{equation*}
        \frac{\partial \log(f^{Gumbel}(X;\mu,\sigma))}{\partial \mu} = \frac{1}{\sigma}\Big(1-\exp\Big(-\frac{x-\mu}{\sigma}\Big)\Big),
    \end{equation*}
    \begin{equation*}
        \frac{\partial \log(f^{Gumbel}(X;\mu,\sigma))}{\partial \sigma} = \frac{x-\mu}{\sigma^2}\Big(1-\exp\Big(-\frac{x-\mu}{\sigma}\Big)\Big)-\frac{1}{\sigma}.
    \end{equation*}
    Notice that, if $X\sim Gumbel(\mu,\sigma)$, for all $k,n\in \mb{N}$
    \begin{equation*}
        \mb{E}\Big[\Big(\frac{X-\mu}{\sigma}\Big)^n \exp\Big(-k\Big(\frac{X-\mu}{\sigma}\Big)\Big)\Big] 
        = (-1)^n\int_0^{+\infty} y^k \log^n(y)\exp(-y)dy = (-1)^n\Gamma^{(n)}(k+1).
    \end{equation*}
    Finally, using this identity, the expressions of the scores and the tabulated values of the Gamma function and its derivatives gives the result.
\end{proof}


\subsection{Two-parameter GEV distribution}\label{sec:GEV2_fisher}
The Fisher information of the $GEV_2(\sigma,\xi)$ distribution are
    \begin{equation*}
        i_{\sigma\sigma} = \Big(\frac{1}{\xi \sigma}\Big)^2, \quad\quad i_{\sigma\xi} =-\frac{1-\gamma}{\sigma\xi^2}-\frac{1}{\sigma\xi^3}, \quad\quad i_{\xi\xi} = \frac{1}{\xi^2}\Big(\frac{\pi^2}{6}+1-2\gamma+\gamma^2\Big)+\frac{2(1-\gamma)}{\xi^3}+\frac{1}{\xi^4},
    \end{equation*}
for $\xi \ne 0$, and $\sigma>0$.

\begin{proof}[Proof]
     Assume $\xi > 0$ for sake of clarity (the case $\xi <0$ uses exactly the same maths). To obtain the Fisher information, we need to compute
    \begin{equation*}
        i_{\sigma\sigma} = \mb{E} \Big[\Big(\frac{\partial \log(f^{GEV_2}(X;\sigma,\xi))}{\partial \sigma}\Big)^2\Big]; \quad i_{\xi\xi} = -\mb{E} \Big[\Big(\frac{\partial \log(f^{GEV_2}(X;\sigma,\xi))}{\partial \xi}\Big)^2\Big];
        \end{equation*}
        \begin{equation*}
        i_{\xi\sigma} = -\mb{E} \Big[\frac{\partial \log(f^{GEV_2}(X;\sigma,\xi))}{\partial \sigma}\frac{\partial \log(f^{GEV_2}(X;\sigma,\xi))}{\partial \xi}\Big],
    \end{equation*}
    where $X\sim GEV_2(\sigma,\xi)$.
    
    \noindent The scores of the $GEV_2$ distribution are

    \begin{equation*}
        \frac{\partial\log(f^{GEV_2}(x;\sigma,\xi))}{\partial \sigma} = \frac{1}{\xi\sigma}\Big(1-\Big(\frac{\xi x}{\sigma}\Big)^{-1/\xi}\Big),
    \end{equation*}

    \begin{equation*}
        \frac{\partial\log(f^{GEV_2}(x;\sigma,\xi))}{\partial \xi} = \frac{1}{\xi^2}\Big(\log\Big(\frac{\xi x}{\sigma}\Big)-1\Big)\Big(1-\Big(\frac{\xi x}{\sigma}\Big)^{-1/\xi}\Big)-\frac{1}{\xi},
    \end{equation*}
for $x> 0$ if $\xi >0$, and for $x<0$ if $\xi <0$.
    
\noindent Notice that, if $X\sim GEV_2(\sigma,\xi)$, for all $k,n\in \mb{N}$
\begin{equation*}
    \mb{E}\Big[\Big(\frac{\xi X}{\sigma}\Big)^{-k/\xi}\log^n\Big(\frac{\xi x}{\sigma}\Big)\Big] 
    =(-\xi)^n\int_0^\infty y^k \log^n(y) \exp(-y) dy = (-\xi)^n\Gamma^{(n)}(k+1).
\end{equation*}
Finally, using this identity, the expressions of the scores given and the tabulated values of the Gamma function and its derivatives gives the results.

\end{proof}

\subsection{GP distribution} \label{sec:fisher_gpd}
The Fisher information matrix for the two-parameter GP distribution was derived in \cite{smith1984threshold}, and expressions for the three-parameter case are given in \cite{taylor2019clustering}. We recall them here for completeness
\begin{equation*}
    i_{\sigma\sigma}=\frac{1}{\sigma^2(1+2\xi)}; \quad i_{\xi\xi}=\frac{2}{(1+\xi)(1+2\xi)}; \quad i_{\sigma\xi}=\frac{1}{\sigma(1+\xi)(1+2\xi)}.
\end{equation*}

\begin{equation*}
    i_{\mu\mu}=\frac{(\xi+1)^2}{\sigma^2(1+2\xi)}; \quad i_{\mu\sigma}=-\frac{\xi}{\sigma^2(1+2\xi)}; \quad i_{\mu\xi}=-\frac{\xi}{\sigma(1+\xi)(1+2\xi)}.
\end{equation*}

Because the all elements of Fisher information are independent of $\mu$, the components $i_{\sigma\sigma},i_{\xi\xi}$ and $i_{\sigma\xi}$ are the same in the three-parameter or two-parameter cases.

\section{Proofs of Section~\ref{sec:main}}\label{sec:appendix_proofs}

\begin{proof}[Proof of Theorem~\ref{thm:Gumbel}]
According to Equation~\eqref{eq:orthogonal_pde}, a reparametrisation with a transformed location parameter $(\nu,\sigma) \mapsto (\Tilde{\mu}(\nu,\sigma),\sigma)$, needs to fulfil
\begin{equation*}
    i_{\Tilde{\mu}\Tilde{\mu}} \frac{\partial \Tilde{\mu}(\nu,\sigma)}{\partial \sigma}= - i_{\Tilde{\mu}\sigma}.
\end{equation*}
Given the Fisher information in Appendix~\ref{sec:gumbel_fisher}, this equation rewrites as
\begin{equation*}
    \frac{\partial \Tilde{\mu}(\nu,\sigma)}{\partial \sigma} = 1- \gamma.
\end{equation*}
A solution of this equation is given by
\begin{equation*}
\Tilde{\mu}(\nu,\sigma) = (1- \gamma)\sigma + C_1(\nu),
\end{equation*}
where $C_1(\nu)$ is independent of $\sigma$.

Similarly, a reparametrisation with a transformed scale  $(\mu,\rho) \mapsto (\mu,\Tilde{\sigma}(\mu,\rho))$, needs to fulfil
\begin{equation*}
    i_{\Tilde{\sigma}\Tilde{\sigma}} \frac{\partial \Tilde{\sigma}(\mu,\rho)}{\partial \mu}= - i_{\mu\Tilde{\sigma}}.
\end{equation*}
Given the Fisher information in Appendix~\ref{sec:gumbel_fisher}, this equation rewrites as
\begin{equation*}
    \frac{\partial \Tilde{\sigma}(\mu,\rho)}{\partial \mu} = \frac{1- \gamma}{\pi^2/6+\gamma^2 - 2 \gamma +1}.
\end{equation*}
A solution of this equation is given by
\begin{equation*}
\Tilde{\sigma}(\mu,\rho) = \frac{1- \gamma}{2c_1 -2\gamma+1}\mu + C_2(\rho),
\end{equation*}
where $C_2(\rho)$ is independent of $\mu$.
\end{proof}

\begin{proof}[Proof of Theorem~\ref{thm:2-GEV}]
    According to Equation~\eqref{eq:orthogonal_pde}, a reparametrisation with a transformed location parameter $(\sigma,\xi) \mapsto (\Tilde{\sigma}(\rho,\xi),\xi)$, needs to fulfil
    \begin{equation*}
        i_{\Tilde{\sigma}\Tilde{\sigma}} \frac{\partial \Tilde{\sigma}(\rho,\xi)}{\partial \xi}= - i_{\Tilde{\sigma}\xi}.
    \end{equation*}
   Given the Fisher information in Appendix~\ref{sec:GEV2_fisher}, this equation rewrites as
    \begin{equation*}
        \frac{\partial \Tilde{\sigma}(\rho,\xi)}{\partial \xi}=  \frac{\Tilde{\sigma}(\rho,\xi)}{\xi}\Big(1-\gamma+\frac{1}{\xi}\Big).
    \end{equation*}
    A solution of this equation is given by
    \begin{equation*}
        \Tilde{\sigma}(\rho,\xi) = C(\rho)\xi\exp\big((1-\gamma)\xi\big),
    \end{equation*}
    where $C(\rho)$ is independent of $\xi$.
\end{proof}

\begin{proof}[Proof of Remark~\ref{rem:GPD}]

\textit{The two-parameter GP distribution.} According to Equation~\eqref{eq:orthogonal_pde}, a reparametrisation with a transformed scale parameter $(\nu,\xi)\mapsto(\Tilde{\sigma}(\nu,\xi),\xi)$ needs to fulfil
\begin{equation*}
i_{\Tilde\sigma\Tilde\sigma}\frac{\partial\Tilde\sigma(\nu,\xi)}{\partial \xi} = -i_{\Tilde\sigma\xi}.
\end{equation*}
Given the Fisher information in Appendix~\ref{sec:fisher_gpd}, this equation rewrites as
\begin{equation*}
\frac{\partial\Tilde\sigma(\nu,\xi)}{\partial \xi} = -\frac{\Tilde\sigma(\nu,\xi)}{1+\xi}.
\end{equation*}
A solution of this equation is given by
\begin{equation*}
    \Tilde{\sigma}(\nu,\xi) = C_3(\nu)/(\xi+1)
\end{equation*}
where $C_3(\nu)$ is independent of $\xi$.

Similarly, a reparametrisation with transformed shape parameter $(\sigma,\zeta)\mapsto(\sigma,\Tilde{\xi}(\sigma,\xi))$ needs to fulfil
\begin{equation*}
    \frac{\partial\Tilde\xi(\sigma,\zeta)}{\partial \sigma} = -\frac{1}{2\sigma}.
\end{equation*}
A solution of this equation is given by
\begin{equation*}
    \Tilde\xi(\sigma,\zeta) = C_3(\zeta)-\log(\sigma)/2,
\end{equation*}
where $C_3(\zeta)$ is independent of $\sigma$.

\textit{The three-parameter GP distribution.} 
According to Equation~\eqref{eq:orthogonal_pde}, a reparametrisation with a transformed location and scale parameters $(\mu,\nu,\xi)\mapsto(\Tilde{\mu}(\mu,\nu,\xi),\Tilde{\sigma}(\mu,\nu,\xi),\xi)$ needs to fulfil
\begin{equation*}
    i_{\Tilde{\mu} \Tilde{\mu}} \frac{\partial \Tilde{\mu}(\mu,\nu,\xi)}{\partial \xi} + i_{\Tilde{\mu} \Tilde{\sigma}}\frac{\partial \Tilde{\sigma}(\mu,\nu,\xi)}{\partial \xi} = -i_{\xi \Tilde{\mu}} , \mbox{ and }\quad
 i_{\Tilde{\mu} \Tilde{\sigma}} \frac{\partial \Tilde{\mu}(\mu,\nu,\xi)}{\partial \xi} + i_{\Tilde{\sigma}\Tilde{\sigma}}\frac{\partial \Tilde{\sigma}(\mu,\nu,\xi)}{\partial \xi} = -i_{\xi \Tilde{\sigma}},
\end{equation*}
resulting in the equations
\begin{equation*}
    (i_{\Tilde{\mu} \Tilde{\sigma}}^2 - i_{\Tilde{\sigma}\Tilde{\sigma}}i_{\Tilde{\mu} \Tilde{\mu}})\frac{\partial \Tilde{\sigma}(\mu,\nu,\xi)}{\partial \xi} = i_{\xi \Tilde{\sigma}}i_{\Tilde{\mu} \Tilde{\mu}} - i_{\xi \Tilde{\mu}}i_{\Tilde{\mu} \Tilde{\sigma}}, \mbox{ and }\quad
    (i_{\Tilde{\mu} \Tilde{\sigma}}^2 - i_{\Tilde{\sigma}\Tilde{\sigma}}i_{\Tilde{\mu} \Tilde{\mu}})\frac{\partial \Tilde{\mu}(\mu,\nu,\xi)}{\partial \xi} =  i_{\xi \Tilde{\mu}}i_{\Tilde{\sigma}\Tilde{\sigma}} - i_{\xi \Tilde{\sigma}}i_{\Tilde{\mu} \Tilde{\sigma}}.
\end{equation*}
Given the Fisher information given in Appendix~\ref{sec:fisher_gpd}, these equations rewrite as
\begin{equation*}
\frac{\partial\Tilde\sigma(\mu,\nu,\xi)}{\partial \xi} = -\frac{\Tilde\sigma(\mu,\nu,\xi)}{1+\xi}, \mbox{ and }\quad \frac{\partial\Tilde\mu(\mu,\nu,\xi)}{\partial \xi} = 0,
\end{equation*}
Solutions of these equations are given by
\begin{equation*}
\Tilde\sigma(\mu,\nu,\xi) = C_5(\mu,\nu)/(\xi+1), \mbox{ and} \quad \Tilde\mu(\mu,\nu,\xi) = C_6(\mu,\nu),
\end{equation*}
where $C_5(\mu,\nu)$ and $C_6(\mu,\nu)$ are independent of $\xi$.
\end{proof}
\end{document}